\documentclass{amsart}
\usepackage{amsmath, amssymb, mathrsfs, verbatim, multirow}
\usepackage[pdftex,
            pdfauthor={M. Ghasemi, M. Marshall},
            pdftitle={Lower Bounds for a Polynomial on a basic closed semialgebraic set using geometric programming},
            pdfsubject={polynomial optimization},
            pdfkeywords={Positive Polynomials, Sums of Squares, Geometric Programming},
            pdfproducer={Latex with hyperref},
            pdfcreator={PDFLaTeX},
            colorlinks=true]{hyperref}
\newcommand{\mathsym}[1]{{}}

\newtheorem{thm}{Theorem}[section]
\newtheorem{lemma}[thm]{Lemma}

\theoremstyle{definition}

\newtheorem{rem}[thm]{Remark}

\begin{document}

\title[Lower bound for a polynomial]
 {Lower Bound for a Polynomial on a product of hyperellipsoids using geometric programming}
 \date{}

\author{Mehdi Ghasemi}
\author{Murray Marshall}
\address{Department of Mathematics and Statistics,
University of Saskatchewan,
Saskatoon, \newline \indent
SK S7N 5E6, Canada}
\email{mehdi.ghasemi@usask.ca, marshall@math.usask.ca}


\begin{abstract} Let $f$ be a polynomial in $n$ variables $x_1,\dots,x_n$ with real coefficients. In \cite{gha-mar2} Ghasemi and Marshall give an algorithm, based on geometric programming, which computes a lower bound for $f$ on $\mathbb{R}^n$. In \cite{gha-lass-mar} Ghasemi, Lasserre and Marshall show how the algorithm in \cite{gha-mar2} can be modified to compute a lower bound for $f$ on the hyperellipsoid $\sum_{i=1}^n x_i^d \le M.$ Here  $d$ is a fixed even integer, $d \ge \max\{ 2, \deg(f)\}$ and $M$ is a fixed positive real number. Suppose now that $g_j := 1-\sum_{i\in I_j} (\frac{x_i}{N_i})^d$, $j=1,\dots,m$, where $d$ is a fixed even integer $d \ge \max\{ 2, \deg(f)\}$, $N_i$ is a fixed positive real number, $i=1,\dots,n$ and  $I_1,\dots, I_m$ is a fixed partition of $\{ 1,\dots,n\}$. The present paper gives an algorithm based on geometric programming for computing a lower bound for $f$ on the subset of $\mathbb{R}^n$ defined by the inequalities $g_j\ge 0$, $j=1,\dots,m$. The algorithm is implemented in a SAGE program developed by the first author. The bound obtained is typically not as sharp as the bound obtained using semidefinite programming, but it has the advantage that it is computable rapidly, even in cases where the bound obtained by semidefinite programming is not computable. When $m=1$ and $N_i = \root d \of{M}$, $i=1,\dots,n$ the algorithm produces the lower bound obtained in \cite{gha-lass-mar}. When $m=n$ and $I_j = \{ j \}$, $j=1,\dots,n$ the algorithm produces a lower bound for $f$ on the hypercube $\prod_{i=1}^n [-N_i,N_i]$, which in certain cases can be computed by a simple formula.
\end{abstract}

\maketitle
%
\section{Introduction}
%
Let $f,g_1,\dots,g_m$ be elements of the polynomial ring $\mathbb{R}[\mathbf{x}]=\mathbb{R}[x_1,\cdots,x_n]$ and let
\[
	K_{\mathbf{g}} := \{ \mathbf{x} \in \mathbb{R}^n : g_j(\mathbf{x})\ge 0, \, j=1,\dots,m\}.
\]
Here,  $\mathbf{g} := (g_1,\dots, g_m)$. We refer to $K_{\mathbf{g}}$ as the basic closed semialgebraic set generated by
$\mathbf{g}$. Observe that if $m=0$, then $\mathbf{g}=\emptyset$ and $K_{\mathbf{g}}= \mathbb{R}^n$. Let
\[
	f_{*,\mathbf{g}}:= \inf\{ f(\mathbf{x}) : \mathbf{x} \in K_{\mathbf{g}}\}.
\]
One would like to have a simple algorithm for computing a lower bound for $f$ on $K_{\mathbf{g}}$, i.e., a lower bound
for $f_{*,\mathbf{g}}$. Lasserre's algorithm \cite{lasserre1} is such an algorithm. It produces a hierarchy of lower bounds
\[
	f_{\operatorname{sos},\mathbf{g}}^{(t)} = \sup \{ r \in \mathbb{R} : f-r = \sum_{j=0}^m \sigma_jg_j, \sigma_j \in \sum \mathbb{R}[\mathbf{x}]^2, \deg(\sigma_jg_j)\le t, j=0,\dots,m \}
\]
for $f$ on $K_{\mathbf{g}}$, one for each integer $t\ge \max\{ \deg(f), \deg(g_j): j=1,\dots,m\}$, which are computable
by semidefinite programming. Here, $g_0:=1$ and $\sum \mathbb{R}[\mathbf{x}]^2$ denotes the set of elements of
$\mathbb{R}[\mathbf{x}]$ which are sums of squares. There is also a variant of the algorithm in \cite{lasserre1}, due to Waki et al, which is useful in certain cases; see \cite{WKKM}.

The algorithm in \cite{gha-mar2} deals with the case $m=0$, producing a lower bound $f_{\operatorname{gp}}$ for $f$
on $\mathbb{R}^n$ computable by geometric programming.\footnotemark\footnotetext{A function $\phi: (0,\infty)^n\rightarrow\mathbb{R}$ of the form $\phi(\underline{x})=cx_1^{a_1}\cdots x_n^{a_n}$,
where $c>0$, $a_i\in\mathbb{R}$ and $\underline{x}=(x_1,\ldots,x_n)$ is called a \it monomial function. \rm A sum of monomial functions
is called a \it posynomial function. \rm An optimization problem of the form
\[
	\left\lbrace
	\begin{array}{ll}
		\textrm{Minimize } & \phi_0(\underline{x})  \\
		\textrm{Subject to} & \phi_i(\underline{x})\leq 1, \ i=1,\ldots,m \text{ and } \  \psi_j(\underline{x})=1, \ j=1,\ldots,p
	\end{array}
	\right.
\]
where $\phi_0,\ldots,\phi_m$ are posynomials and $\psi_1,\ldots,\psi_p$ are monomial functions, is called a \textit{geometric program}. See \cite{boyd} and \cite{gha-mar2}.}
See \cite{fidalgo}, \cite{gha-mar1} and \cite{lasserre} for
precursors of \cite{gha-mar2}. See \cite{DT} and \cite{IW} for extensions of these results.  The algorithm in \cite{gha-lass-mar} is a variation of the algorithm in \cite{gha-mar2}
which deals with the case $m=1$, $g_1 = M-\sum_{i=1}^n x_i^d$, i.e., it produces a lower bound for $f$ on the
hyperellipsoid $\sum_{i=1}^n x_i^d\le M$. Here, $d$ is an even integer, $d \ge \max\{ 2, \deg(f)\}$. Again, this lower bound is computable by geometric programming.

Although the bounds obtained in \cite{gha-lass-mar} and \cite{gha-mar2} are typically not as sharp as the bounds obtained in \cite{lasserre1} and \cite{WKKM}, the computation is much faster, especially when the coefficients are sparse, and problems where the number of variables and the degree are large (problems where the methods in \cite{lasserre1} and \cite{WKKM} break down completely) can be handled easily.

For example, the exact minimum of $f=x_1^{40}+x_2^{40}+x_3^{40}-x_1x_2x_3$ on $\mathbb{R}^3$ can be computed in about $0.18$ seconds, using the geometric programming method.
But even for such a simple example, the methods in \cite{lasserre1} and \cite{WKKM} break down completely; see \cite[Example 3.9]{gha-mar2}.

In the present paper we establish a lower bound for $f$ on $K_{\mathbf{g}}$, computable by geometric
programming, in case $g_j := 1-\sum_{i\in I_j} (\frac{x_i}{N_i})^d$, $j=1,\dots,m$, where $d$ is a fixed even integer $d \ge \max\{ 2, \deg(f)\}$, the $N_i$ are fixed positive real numbers, and  $I_1,\dots, I_m$ is a fixed partition of $\{ 1,\dots,n\}$; see Theorem \ref{the algorithm}. As in \cite{gha-lass-mar} and \cite{gha-mar2}, the algorithm exploits automatically any sparsity of coefficients.
In case $m=1$ and $N_i = \root d \of{M}$, $i=1,\dots, n$, the bound coincides with the one obtained in \cite{gha-lass-mar}, although the algorithm is not quite the same. (Our algorithm is actually  simpler.)

See Tables 1, 2 and 3 for runtime and relative error computations. The source code of a SAGE program, developed by the first author, which computes the lower bound described in Theorem \ref{the algorithm}, is available at \href{https://github.com/mghasemi/CvxAlgGeo}{github.com/mghasemi/CvxAlgGeo}.

In the special case where $m=n$ and $I_j = \{ j \}$, $j=1,\dots,n$, $K_{\bf g}$ coincides with the hypercube $\prod_{j=1}^n [-N_j,N_j]$. Theorem \ref{hypercube case} relates the lower bound obtained by applying our algorithm in this special case to a variant of the trivial lower bound introduced in \cite[Section 3]{gha-lass-mar}.

It needs to be recorded, finally, that the present paper is nothing more or less than a shortened more focused version of the paper \cite{gha-mar3}. The reader may wish to look at the latter paper, where the algorithm described in the present paper is extended to cover a large number of additional cases.
\section{the case $m=0$}
We recall the algorithm established in \cite{gha-mar2}. We need some notation. Fix an even integer $d\ge 2$. Let
$\mathbb{N}^n_d:=\{\alpha\in\mathbb{N}^n:\vert\alpha\vert\leq d\}$ where $\vert\alpha\vert=\sum_i\alpha_i$ for every
$\alpha\in\mathbb{N}^n$. Let $\epsilon_i:=(\epsilon_{i1},\cdots,\epsilon_{in})\in\mathbb{N}^n$, where $\epsilon_{ij}= \begin{cases} d \text{ if } i=j \\ 0 \text{ if } i \ne j \end{cases}$ and, given $f= \sum f_{\alpha} \mathbf{x}^{\alpha}\in\mathbb{R}[\mathbf{x}]$, $\deg(f)\le d$, let:
\[
\begin{array}{lcl}
\Omega(f)&:=&\{\alpha\in\mathbb{N}^n_{d}\::\:f_\alpha\neq0\}\setminus\{0,\epsilon_1,\cdots,\epsilon_n\}\\
\Delta(f)&:=&\{\alpha\in\Omega(f)\::\:f_\alpha\,\mathbf{x}^\alpha\mbox{ is not a square in }\mathbb{R}[\mathbf{x}]\}\\
\Delta(f)^{< d}&:=&\{\alpha\in\Delta(f)\::\:\vert\alpha\vert <d\}\\
\Delta(f)^{= d}&:=&\{\alpha\in\Delta(f)\::\:\vert\alpha\vert =d\}.
\end{array}
\]
Denote the coefficient $f_{\epsilon_i}$ by $f_{d,i}$ for $i=1,\dots,n$. One is most interested in the case where $\deg(f)=d$.
\begin{thm} \cite[Theorem 3.1]{gha-mar2}
\label{thm1}
Let $f\in\mathbb{R}[\mathbf{x}]$, $\deg(f)\le d$, and let $\rho(f)$ denote the optimal value of the program:
\begin{equation}
\label{lw}
\left\{
\begin{array}{lll}
Minimize & \sum\limits_{\alpha\in\Delta(f)^{<d}}(d-\vert\alpha\vert)
\left[\left(\frac{f_\alpha}{d}\right)^{d}\,\left(\frac{\alpha}{\mathbf{z}_\alpha}\right)^\alpha\right]^{1/(d-\vert\alpha\vert)} & \\
\mbox{s.t.} & \sum\limits_{\alpha\in\Delta(f)}z_{\alpha,i}\le f_{d,i}, & i=1,\ldots,n\\
& \left(\frac{\mathbf{z}_\alpha}{\alpha}\right)^\alpha =\left(\frac{f_\alpha}{d}\right)^{d}, & \alpha\in\Delta(f)^{=d}
\end{array}\right.
\end{equation}
where, for every $\alpha\in\Delta(f)$, the unknowns $\mathbf{z}_\alpha = (z_{\alpha,i})\in[0,\infty)^n$ satisfy $z_{\alpha,i}= 0$
if and only if $\alpha_i=0$. Then $f-f(0)+\rho(f)$ is a sum of binomial squares. In particular,
$f_{\operatorname{gp}} := f(0)-\rho(f)$ is a lower bound for $f$ on $\mathbb{R}^n$.
\end{thm}
Here, $\left(\frac{\alpha}{\mathbf{z}_\alpha}\right)^\alpha := \prod_{i=1}^n \frac{\alpha_i^{\alpha_i}}{(z_{\alpha,i})^{\alpha_i}}$
and $\left(\frac{\mathbf{z}_\alpha}{\alpha}\right)^\alpha := \prod_{i=1}^n \frac{(z_{\alpha,i})^{\alpha_i}}{\alpha_i^{\alpha_i}}$,
the convention being that $0^0 = 1$.

If the feasible set of the program (\ref{lw}) is empty, then $\rho(f) = \infty$ and $f_{\operatorname{gp}} = -\infty$.
A sufficient (but not necessary) condition for the feasible set of (\ref{lw}) to be nonempty is that $\Delta(f)^{=d} =\emptyset$
and $f_{d,i}>0$, $i=1,\dots,n$. If $\deg(f)<d$ then either $\Delta(f)=\emptyset$ and $f_{\operatorname{gp}}=f(0)$ or
$\Delta(f)\ne \emptyset$ and $f_{\operatorname{gp}}=-\infty$.

If $f_{d,i}>0$, $i=1,\dots,n$ then (\ref{lw}) is a geometric program. Somewhat more generally, if $\forall$ $i=1,\dots,n$
either ($f_{d,i}>0$) or ($f_{d,i}=0$ and  $\alpha_i=0$ $\forall$ $\alpha \in \Delta(f)$), then (\ref{lw}) is a geometric program.
In the remaining cases (\ref{lw}) is not a geometric program and the feasible set of (\ref{lw}) is empty.
\section{lower bound on a product of hyperellipsiods}

Fix an even integer $d \ge 2$ and suppose $f, g_1,\dots,g_m$ have degree at most $d$. Let $G(\lambda) := f-\sum_{j=1}^m \lambda_jg_j$ where
$\lambda = (\lambda_1,\dots,\lambda_m) \in [0,\infty)^m$. By Theorem \ref{thm1}, $G(\lambda)_{\operatorname{gp}}$ is a lower bound for $G(\lambda)$ on $\mathbb{R}^n$. It follows that $G(\lambda)_{\operatorname{gp}}$ is a lower bound for $f$ on $K_{\mathbf{g}}$ and consequently, that
\[
	s(f,\mathbf{g}):=\sup\{ G(\lambda)_{\operatorname{gp}} : \lambda \in [0,\infty)^m\}
\]
is a lower bound for $f$ on $K_{\mathbf{g}}$.\footnote{In fact, $s(f,\mathbf{g}) \le f_{\operatorname{sos},\mathbf{g}}^{(d)}$.
By Theorem \ref{thm1}, $G(\lambda)-G(\lambda)_{\operatorname{gp}}$ is a sum of binomial squares (obviously of degree at most $d$)
for each $\lambda \in [0,\infty)^m$. This implies that $G(\lambda)_{\operatorname{gp}} \le f_{\operatorname{sos},\mathbf{g}}^{(d)}$
for each $\lambda \in [0,\infty)^m$.}  By Theorem \ref{thm1}, for  each
$\lambda \in [0,\infty)^m$, $G(\lambda)_{\operatorname{gp}}$ is computable by geometric programming.
Unfortunately, this does not imply that the supremum is so computable, but there are certain special cases where it is \cite{gha-mar3}. In particular, we have the following: 

\begin{thm} \label{the algorithm} Assume $m\ge 1$ and let $g_j := 1-\sum_{i\in I_j} (\frac{x_i}{N_i})^d$, $j=1,\dots,m$, where $d$ is a fixed even integer $d \ge \max\{ 2, \deg(f)\}$, $N_i$ is a fixed positive real number, $i=1,\dots, n$ and  $I_1,\dots, I_m$ is a fixed partition of $\{ 1,\dots,n\}$. Let $c_j := \max\{ f_{d,i}N_i^d, 0 : i \in I_j\}$,  $j=1,\dots, m.$ Then $s(f,{\bf g}) = f(0)+\sum_{j=1}^m c_j - \rho$, where $\rho$ denotes the optimum value of the geometric program:
\begin{equation}
\label{lw3}
\left\{\begin{array}{lll}
	\textrm{Minimize} & \multicolumn{2}{c}{\sum\limits_{j=1}^m \mu_j+ \sum\limits_{\alpha\in\Delta(f)^{<d}}(d-\vert\alpha\vert)\left[
	 \left(\frac{f_\alpha}{d}\right)^{d}\,\left(\frac{\alpha}{\mathbf{z}_\alpha}\right)^\alpha
	\right]^{1/(d-\vert\alpha\vert)}}\\
	\textrm{\mbox{s.t.}} & \sum\limits_{\alpha\in\Delta(f)}z_{\alpha,i}+(\frac{c_j}{N_i^d}-f_{i,d}) \le \frac{\mu_j}{N_i^d}, & i \in I_j, \ j=1,\ldots,m\\
	& \left(\frac{\mathbf{z}_\alpha}{\alpha}\right)^\alpha = \left(\frac{f_\alpha}{d}\right)^{d}, & \alpha\in\Delta(f)^{=d}\\
	& c_j \le \mu_j, & j=1,\dots,m
\end{array}\right.
\end{equation}
Here, for every $\alpha\in\Delta(f)$, the unknowns $\mathbf{z}_\alpha = (z_{\alpha,i})\in[0,\infty)^n$ satisfy $z_{\alpha,i}= 0$ if and only if $\alpha_i=0$, and the unknowns $\mu = (\mu_1,\dots,\mu_m)$ satisfy $\mu_j> 0$.
\end{thm}

\begin{proof} Note that $G(\lambda)_{\alpha} = f_{\alpha}$ for $\alpha \in \Delta(f)$,
$G(\lambda)_{d,i} = f_{d,i}+\frac{\lambda_j}{N_i^d}$, for $i \in I_j$, and $G(\lambda)(0) = f(0)-\sum_{j=1}^m \lambda_j$.

Suppose first that $({\bf z},\mu)$ is a feasible point for the program (\ref{lw3}). Let $\lambda = (\lambda_1,\dots,\lambda_m)$ where $\lambda_j := \mu_j-c_j$, $j=1,\dots,m$. Then $\lambda \in [0,\infty)^m$ and ${\bf z}$ is a feasible point for the program:
\begin{equation}\label{lw'}
\left\{
\begin{array}{lll}
	\textrm{Minimize} & \multicolumn{2}{c}{\sum\limits_{\alpha\in\Delta(f)^{<d}}(d-\vert\alpha\vert)
	 \left[\left(\frac{f_\alpha}{d}\right)^{d}\,\left(\frac{\alpha}{\mathbf{z}_\alpha}\right)^\alpha
	\right]^{1/(d-\vert\alpha\vert)}}\\
	\textrm{\mbox{s.t.}} & \sum\limits_{\alpha\in\Delta(f)}z_{\alpha,i} \le f_{i,d}+\frac{\lambda_j}{N_i^d}, & i  \in I_j, \ j=1,\ldots,m \\
	& \left(\frac{\mathbf{z}_\alpha}{\alpha}\right)^\alpha =\left(\frac{f_\alpha}{d}\right)^{d}, & \alpha\in\Delta(f)^{=d}
\end{array}\right.
\end{equation}
It follows that
\begin{align*}
& s(f,{\bf g}) \ge G(\lambda)_{\operatorname{gp}}  = f(0)- \sum_{j=1}^m \lambda_j - \rho(G(\lambda)) \\ & \ge f(0)- \sum_{j=1}^m \lambda_j - \sum\limits_{\alpha\in\Delta(f)^{<d}}(d-\vert\alpha\vert)\left[
	 \left(\frac{f_\alpha}{d}\right)^{d}\,\left(\frac{\alpha}{\mathbf{z}_\alpha}\right)^\alpha
	\right]^{1/(d-\vert\alpha\vert)} \\ & = f(0)+ \sum_{j=1}^m c_j -\sum_{j=1}^m \mu_j - \sum\limits_{\alpha\in\Delta(f)^{<d}}(d-\vert\alpha\vert)\left[
	 \left(\frac{f_\alpha}{d}\right)^{d}\,\left(\frac{\alpha}{\mathbf{z}_\alpha}\right)^\alpha
	\right]^{1/(d-\vert\alpha\vert)},
\end{align*}
so $s(f,{\bf g}) \ge f(0)+\sum_{j=1}^m c_j -\rho$.

Now let $\lambda \in [0,\infty)^m$, let $\mathbf{z}$ be a feasible point for the program (\ref{lw'}), and let $\epsilon>0$ be given.
Let $\mu = (\mu_1,\dots,\mu_m)$ where $\mu_j:= \lambda_j+\frac{\epsilon}{m}+c_j$, $j=1,\dots,m$.
Then $({\bf z}, \mu)$ is a feasible point for the program (\ref{lw3}), so
\begin{align*}
	& f(0)+\sum_{j=1}^m c_j -\rho \\ & \ge f(0)+\sum_{j=1}^m c_j- \sum_{j=1}^m \mu_j  -  \sum\limits_{\alpha\in\Delta(f)^{<d}}(d-\vert\alpha\vert)\left[
	 \left(\frac{f_\alpha}{d}\right)^{d}\,\left(\frac{\alpha}{\mathbf{z}_\alpha}\right)^\alpha
	\right]^{1/(d-\vert\alpha\vert)} \\
	 & \ge f(0)-  \sum_{j=1}^m \lambda_j-\sum\limits_{\alpha\in\Delta(f)^{<d}}(d-\vert\alpha\vert)\left[
	 \left(\frac{f_\alpha}{d}\right)^{d}\,\left(\frac{\alpha}{\mathbf{z}_\alpha}\right)^\alpha
	\right]^{1/(d-\vert\alpha\vert)} - \epsilon.
\end{align*}
It follows that $f(0)+\sum_{j=1}^m c_j -\rho \ge  G(\lambda)_{\operatorname{gp}}-\epsilon$. Since this holds for any $\lambda \in [0,\infty)^m$ and for any $\epsilon>0$, it follows that $f(0)+\sum_{j=1}^m c_j -\rho \ge s(f, {\bf g})$.
\end{proof}

Table \ref{RunTimeTable} records average running time for computation of $s(f,\mathbf{g})$ for large examples (where computation of $f_{\operatorname{sos},\mathbf{g}}^{(d)}$ is not possible). Here $N_i :=1$, $i = 1,\dots,n$, so $g_j = 1 - \sum_{i\in I_j} x_i^d$, $j=1,\dots,m$. The average is taken over 10 randomly chosen partitions $\{I_1,\dots,I_m\}$ and polynomials $f$, each $f$ having $t$ terms and $\deg(f)\le d$ with coefficients chosen from $[-10,10]$. See also \cite[Table 2]{gha-lass-mar} and \cite[Table 3]{gha-mar2}.

\begin{table}[ht]{\footnotesize
\caption{Average runtime for computation of $s(f,\mathbf{g})$ (seconds) for various $n, d$ and $t$.}\label{RunTimeTable}
\centering
\begin{tabular}{|c|c|cccc|}
\hline
$n$ & $d\backslash t$ & 50 & 100 & 150 & 200 \\
\hline
\multirow{3}{*}{10} & 20 & 3.330 & 23.761 & 79.369 & 170.521 \\
				    & 40 & 5.730 & 43.594 & 159.282 & 421.497 \\
				    & 60 & 6.524 & 68.625 & 191.126 & 531.442 \\
\hline
\multirow{3}{*}{20} & 20 & 8.364 & 63.198 & 193.431 & 562.243 \\
				    & 40 & 16.353 & 149.137 & 473.805 & 1102.579 \\
				    & 60 & 31.774 & 304.065 & 782.967 & 1184.263 \\
\hline
\multirow{3}{*}{30} & 20 & 12.746 & 107.285 & 353.803 & 776.831 \\
				    & 40 & 46.592 & 310.228 & 753.356 & 1452.159 \\
				    & 60 & 58.838 & 539.738 & 1271.102 & 1134.887 \\
\hline
\multirow{3}{*}{40} & 20 & 15.995 & 148.827 & 423.117 & 989.318 \\
				    & 40 & 60.861 & 414.188 & 1493.461 & 1423.965 \\
				    & 60 & 95.384 & 784.039 & 1305.201 & 1093.932 \\
\hline

\end{tabular}}
\end{table}

Table \ref{runtimecmp} records average running time for computation of $s(f,\mathbf{g})$ (the top number) and $f^{(d)}_{\operatorname{sos},\mathbf{g}}$ (the bottom number) for small examples. Here $g_j = 1 - \sum_{i\in I_j} x_i^d$, $j=1,\dots,m$. The average is taken over 10 randomly chosen partitions $\{I_1,\dots,I_m\}$ and polynomials $f$, each $f$ having $t$ terms and $\deg(f)\le d$ with coefficients chosen from $[-10,10]$.\footnote{\textbf{Hardware and Software
specifications.} Processor: Intel\textregistered~ Core\texttrademark2 Duo CPU P8400 @ 2.26GHz, Memory: 3 GB, OS: \textsc{Ubuntu}
14.04-32 bit, \textsc{Sage}-6.0} See also \cite[Table 1]{gha-lass-mar} and \cite[Table 2]{gha-mar2}.

\begin{table}[ht]{\footnotesize
\caption{Average runtime for computation of $s(f,\mathbf{g})$ and $f^{(d)}_{\operatorname{sos},\mathbf{g}}$
for various $n, d$ and $t$.}\label{runtimecmp}

\centering
\begin{tabular}{|c|c|cccccccccc|}
\hline
$n$ & $d\backslash t$ & 10 & 30 & 50 & 100 & 150 & 200 & 250 & 300 & 350 & 400 \\
\hline
\multirow{6}{*}{3} & \multirow{2}{*}{4} & 0.034 & 0.092 &  &  &  &  &  &  &  & \\
					&				   & 0.026 & 0.026 &  &  &  &  &  &  &  & \\
				  & \multirow{2}{*}{6} & 0.041 & 0.105 &  &  &  &  &  &  &  & \\
				    	&				   & 0.108 & 0.103 &  &  &  &  &  &  &  & \\
				  & \multirow{2}{*}{8} & 0.045 & 0.164 &  &  &  &  &  &  &  & \\
				  	&				   & 0.522 & 0.510 &  &  &  &  &  &  &  & \\
\hline
\multirow{6}{*}{4} & \multirow{2}{*}{4} & 0.042 & 0.116 &  &  &  &  &  &  &  & \\
					&				   & 0.045 & 0.054 &  &  &  &  &  &  &  & \\
				  & \multirow{2}{*}{6} & 0.043 & 0.133 & 0.285 &  &  &  &  &  &  & \\
				    	&				   & 0.662 & 0.632 & 0.624 &  &  &  &  &  &  & \\
				  & \multirow{2}{*}{8} & 0.053 & 0.191 & 0.446 & 2.479 & 7.089 &  &  &  &  & \\
				  	&				   & 12.045 & 13.019 & 11.956 & 12.091 & 12.307 &  &  &  &  & \\
\hline
\multirow{6}{*}{5} & \multirow{2}{*}{4} & 0.039 & 0.124 & 0.295 &  &  &  &  &  &  & \\
					&				   & 0.181 & 0.154 & 0.127 &  &  &  &  &  &  & \\
				  & \multirow{2}{*}{6} & 0.052 & 0.156 & 0.396 & 1.677 & 5.864 & 13.814 &  &  &  & \\
				    	&				   & 6.429 & 6.530 & 6.232 & 6.425 & 6.237 & 6.469 &  &  &  & \\
				  & \multirow{2}{*}{8} & 0.056 & 0.219 & 0.528 & 3.046 & 7.663 & 23.767 & 44.699 & 88.104 & 123.986 & 179.126\\
				  	&				   & 340.321 & 243.860 & 225.746 & 205.621 & 222.619 & 220.887 & 224.018 & 218.994 & 219.085 & 213.348\\
\hline
\multirow{6}{*}{6} & \multirow{2}{*}{4} & 0.043 & 0.140 & 0.340 & 1.545 &  &  &  &  &  & \\
					&				   & 0.453 & 0.422 & 0.422 & 0.448 &  &  &  &  &  & \\
				  & \multirow{2}{*}{6} & 0.053 & 0.194 & 0.429 & 2.079 & 7.138 & 16.212 & 34.464 & 71.465 & 102.342 & 166.345 \\
				    	&				   & 48.239 & 47.752 & 47.776 & 51.268 & 51.008 & 48.012 & 49.908 & 53.135 & 51.542 & 52.874 \\
				  & \multirow{2}{*}{8} & 0.066 & 0.251 & 0.681 & 3.389 & 11.985 & 36.495 & 74.088 & 148.113 & 174.163 & 269.544 \\
				    	&				   & -- & -- & -- & -- & -- & -- & -- & -- & -- & -- \\
\hline
\end{tabular}}
\end{table}

Table \ref{PartitionTable} computes average values for the relative error
\[
	R=  \frac{-s(-f,\mathbf{g})-s(f,\mathbf{g})}{-(-f)_{\operatorname{sos},\mathbf{g}}^{(d)}-f_{\operatorname{sos},\mathbf{g}}^{(d)}}
\]
for small examples. Here $g_j = 1 - \sum_{i\in I_j} x_i^d$, $j=1,\dots,m$.
The average is taken over 20 randomly chosen partitions $\{I_1,\dots,I_m\}$ and  polynomials $f$, each $f$ having $t$ terms and $\deg(f)\le d$ with coefficients chosen from $[-10,10]$. Table \ref{PartitionTable} would seem to confirm that for fixed $n,d$
the quality of the bound $s(f,\mathbf{g})$ is best when $t$ is small, and for fixed $d,t$ the quality of the bound $s(f,\mathbf{g})$
is best when $n$ is large.  Comparison of Table \ref{PartitionTable} with \cite[Table 3]{gha-lass-mar} would seem to indicate
that the quality of the bound $s(f,\mathbf{g})$ is best when $m=1$.
\begin{table}[ht]{\footnotesize
\caption{Average values of $R$ for various $n, d$ and $t$.}\label{PartitionTable}
\centering
\begin{tabular}{|c|c|ccccccccc|}
\hline
$n$ & $d\backslash t$ & 5 & 10 & 50 & 100 & 150 & 200 & 250 & 300 & 400 \\
\hline
\multirow{3}{*}{3} & 4 & 1.3688 & 1.6630 & & & & & & & \\
				   & 6 & 1.5883 & 2.0500 & 4.3726 & & & & & & \\
				   & 8 & 2.0848 & 2.7636 & 5.6391 & 5.8140 & 6.9135 & & & & \\
\hline
\multirow{3}{*}{4} & 4 & 1.2183 & 1.3420 & 3.3995 & & & & & & \\
				   & 6 & 1.3816 & 2.9584 & 3.1116 & 4.6891 & 5.8067 & 6.5150 & & & \\
				   & 8 & 1.7630 & 2.2038 & 3.2685 & 4.4219 & 5.7929 & 7.0841 & 7.9489 & 8.7924 & 9.6068 \\
\hline
\multirow{3}{*}{5} & 4 & 1.2566 & 1.6701 & 3.1867 & 4.6035 & & & & & \\
				   & 6 & 2.3807 & 2.9424 & 3.3077 & 4.7939 & 5.6523 & 7.1996 & 8.6194 & 9.3317 & 10.134\\
				   & 8 & 1.5557 & 1.9754 & 2.5204 & 3.9815 & 4.5404 & 5.1756 & 6.2214 & 6.6919 & 7.8921 \\
\hline
\multirow{3}{*}{6} & 4 & 1.2069 & 1.3876 & 3.0639 & 4.5326 & 4.9645 & 6.1414 & & & \\
				   & 6 & 1.2602 & 1.4854 & 2.8236 & 4.0256 & 4.5797 & 6.3479 & 6.9487 & 7.4866 & 8.8435 \\
				   & 8 & 1.0478 & 1.1616 & 2.4884 & 3.3896 & 4.0870 & 5.0809 & 5.6932 & 6.1994 & 10.567 \\
\hline
\multirow{2}{*}{7} & 4 & 1.1943 & 1.3360 & 2.8592 & 4.5334 & 5.5064 & 5.9837 & 7.4782 & 7.3568 & \\
				   & 6 & 1.2604 & 1.4529 & 2.6962 & 3.8035 & 4.5351 & 6.0305 & 6.1478 & 6.6746 & 9.5049 \\
\hline
\multirow{2}{*}{8} & 4 & 1.1699 & 1.4274 & 2.5942 & 4.0914 & 5.4111 & 6.2111 & 7.4479 & 8.3168 & 9.1369 \\
				   & 6 & 1.0454 & 1.1270 & 2.1316 & 3.2807 & 3.5468 & 4.8955 & 5.1211 & 5.4214 & 7.1872 \\
\hline
\multirow{1}{*}{9} & 4 & 1.2158 & 1.3214 & 2.9305 & 4.0624 & 5.9063 & 6.5985 & 7.9552 & 8.4233 & 11.810 \\
\hline
\multirow{1}{*}{10} & 4 & 1.1476 & 1.3441 & 2.2393 & 3.7136 & 5.7739 & 6.1791 & 6.3211 & 8.7773 & 10.428 \\
\hline
\end{tabular}}
\end{table}
\section{The trivial bound on $\prod_{i=1}^n[-N_i,N_i]$}
Fix $f\in \mathbb{R}[\mathbf{x}]$ and $\mathbf{N} = (N_1,\dots,N_n)$, $N_i>0$, $i=1,\dots,n$. If $\mathbf{x} \in \prod_{i=1}^n [-N_i,N_i]$, then
\[
	f(\mathbf{x}) = \sum f_{\alpha}\mathbf{x}^{\alpha} \ge f(\mathbf{0})-\sum_{\alpha \in \Delta'(f)} |f_{\alpha}|\cdot |\mathbf{x}^{\alpha}| \ge f(\mathbf{0})-\sum_{\alpha\in \Delta'(f)} |f_{\alpha}|\cdot\mathbf{N}^{\alpha},
\]
where $\Delta'(f):= \{ \alpha \in \mathbb{N}^n : |\alpha|>0 \text{ and } f_{\alpha}\mathbf{x}^{\alpha} \text{ is not a square in } \mathbb{R}[\mathbf{x}]\}$
and $\mathbf{N}^{\alpha} := \prod_{i=1}^n N_i^{\alpha_i}$. Set
\begin{equation}
f_{\operatorname{tr},\mathbf{N}} := f(\mathbf{0}) - \sum_{\alpha\in \Delta'(f)} |f_{\alpha}|\cdot \mathbf{N}^{\alpha}.
\end{equation}
Thus $f_{\operatorname{tr},\mathbf{N}}$ is a lower bound for $f$ on the hypercube $\prod_{i=1}^n [-N_i,N_i]$.
We refer to $f_{\operatorname{tr},\mathbf{N}}$ as the \textit{trivial bound} for $f$ on $\prod_{i=1}^n [-N_i,N_i]$.
If $N_i = \root d \of{M}$, $i=1,\dots,n$, this coincides with the trivial bound defined in \cite[Section 3]{gha-lass-mar}.

Suppose now that $d$ is an even integer, $d \ge \max\{ 2, \deg f\}$. Define $\mathbf{g} = (1-(\frac{x_1}{N_1})^d,\dots,1-(\frac{x_n}{N_n})^d)$.
We want to compare $s(f,\mathbf{g})$ with $f_{\operatorname{tr},\mathbf{N}}$.
\begin{thm} \label{hypercube case} Set-up as above. Then  \

(1) $s(f,\mathbf{g}) \ge f_{\operatorname{tr},\mathbf{N}}$.

(2) If $f_{d,i}\le 0$ for $i=1,\dots,n$ then $s(f,\mathbf{g}) = f_{\operatorname{tr},\mathbf{N}}$.
In particular, if $\deg f <d$ then  $s(f,\mathbf{g}) = f_{\operatorname{tr},\mathbf{N}}$.
\end{thm}
We remark that if $f_{i,d}>0$ for some $i \in \{ 1,\dots, n\}$ then the bound $s(f,\mathbf{g})$ can strictly greater than $f_{\operatorname{tr},\mathbf{N}}$.
E.g., if $n=1$, $f=x_1^2-x_1$, $d=2$, $N_1=1$,
then $f_{\operatorname{tr},\mathbf{N}} = -1$, $s(f,\mathbf{g})=f_{\operatorname{\operatorname{gp}}}=-\frac{1}{4}$.
\begin{proof}
By making the change of variables $y_i = N_i\cdot x_i$, $i=1,\dots,n$, we are reduced to the case where $N_1=\dots = N_n=1$.
By definition of $\mathbf{g}$,
\[
G(\lambda) = f-\sum_{i=1}^n \lambda_i(1-x_i^d) = f(\mathbf{0})-\sum_{i=1}^n \lambda_i+\sum_{\alpha\in \Omega(f)}f_{\alpha}\mathbf{x}^{\alpha}+\sum_{i=1}^n (f_{d,i}+\lambda_i)x_i^d,
\]
and $s(f,\mathbf{g})$ is obtained by maximizing the objective function
\begin{equation}\label{obj}
f(\mathbf{0})-\sum_{i=1}^n \lambda_i- \sum_{\alpha\in\Delta(f)^{<d}} (d-\vert\alpha\vert)\left[
\left(\frac{f_\alpha}{d}\right)^{d}\,\left(\frac{\alpha}{\mathbf{z}_\alpha}\right)^\alpha
\right]^{1/(d-\vert\alpha\vert)}
\end{equation}
subject to
\begin{equation}\label{lw5}
\left\{
\begin{array}{lr}
	\sum\limits_{\alpha\in\Delta(f)}z_{\alpha,i} \ \le \ f_{d,i}+\lambda_i, & i=1,\ldots,n\\
	 \left(\frac{f_\alpha}{d}\right)^{d}\left(\frac{\alpha}{\mathbf{z}_\alpha}\right)^\alpha  =  1, & \alpha\in\Delta(f)^{=d}
\end{array}\right.
\end{equation}
where, $\lambda_i\ge 0$, $z_{\alpha,i}\ge 0$ and $z_{\alpha,i}= 0$ if and only if $\alpha_i=0$.

(1) Define $z_{\alpha,i}:= \alpha_i\frac{|f_{\alpha}|}{d}$ and
$\lambda_i := \max\{ 0, \sum_{\alpha\in \Delta(f)} z_{\alpha,i} - f_{d,i}\}$. One checks that, for this choice of
$z_{\alpha,i}$ and $\lambda_i$, the constraints of (\ref{lw5}) are satisfied. Observe also that $f_{d,i}\ge 0$ $\Rightarrow$
$\lambda_i \le \sum_{\alpha\in \Delta(f)} z_{\alpha,i}$ and $f_{d,i}<0$ $\Rightarrow$
$\lambda_i = \sum_{\alpha\in \Delta(f)}z_{\alpha,i}-f_{d,i}$. Consequently,
\begin{align*}
\sum_{i=1}^n \lambda_i \le& \sum_{f_{d,i}\ge 0} (\sum_{\alpha\in \Delta(f)} z_{\alpha,i})+ \sum_{f_{d,i}<0}(\sum_{\alpha\in\Delta(f)}z_{\alpha,i}-f_{d,i})\\
=& \sum_{i=1}^n (\sum_{\alpha\in \Delta(f)} z_{\alpha,i})+\sum_{f_{d,i}<0} |f_{d,i}|\\
=& \sum_{\alpha\in \Delta(f)}(\sum_{i=1}^n z_{\alpha,i})+\sum_{f_{d,i}<0} |f_{d,i}| \\
=& \sum_{\alpha\in \Delta(f)}|\alpha|\cdot \frac{|f_{\alpha}|}{d} +\sum_{f_{d,i}<0} |f_{d,i}|,
\end{align*}
and
\begin{align*}
s(f,\mathbf{g}) \ge& f(\mathbf{0})-\sum_{i=1}^n \lambda_i- \sum_{\alpha\in\Delta(f)^{<d}} (d-\vert\alpha\vert)\left[
\left(\frac{f_\alpha}{d}\right)^{d}\,\left(\frac{\alpha}{\mathbf{z}_\alpha}\right)^\alpha
\right]^{1/(d-\vert\alpha\vert)}\\
\ge& f(\mathbf{0})- \sum_{\alpha\in \Delta(f)}|\alpha|\cdot \frac{|f_{\alpha}|}{d} -\sum_{f_{d,i}<0} |f_{d,i}|-
\sum_{\alpha\in \Delta(f)^{<d}}(d-|\alpha|)\cdot\frac{|f_{\alpha}|}{d}\\
=&f(\mathbf{0})-\sum_{\alpha\in \Delta(f)} |f_{\alpha}|  -\sum_{f_{d,i}<0} |f_{d,i}|\\
=& f_{\operatorname{tr},\mathbf{N}}.
\end{align*}

(2) Suppose $(\mathbf{z},\lambda)$ satisfies (\ref{lw5}). Since we are trying to maximize (\ref{obj}), we may as well assume each
$\lambda_i$ is chosen as small as possible, i.e., $\lambda_i = \max\{ 0, \sum_{\alpha\in \Delta(f)} z_{\alpha,i}-f_{d,i}\}$.
Since we are also assuming $f_{d,i}\le 0$, this means $\lambda_i = \sum_{\alpha\in \Delta(f)} z_{\alpha,i}-f_{d,i}$. Then
\begin{align*}
&f(\mathbf{0})-\sum_{i=1}^n \lambda_i- \sum_{\alpha\in\Delta(f)^{<d}} (d-\vert\alpha\vert)\left[
\left(\frac{f_\alpha}{d}\right)^{d}\,\left(\frac{\alpha}{\mathbf{z}_\alpha}\right)^\alpha
\right]^{1/(d-\vert\alpha\vert)}\\
&=f(\mathbf{0})-\sum_{i=1}^n(\sum_{\alpha\in \Delta(f)}z_{\alpha,i}-f_{d,i})- \sum_{\alpha\in\Delta(f)^{<d}} (d-\vert\alpha\vert)\left[
\left(\frac{f_\alpha}{d}\right)^{d}\,\left(\frac{\alpha}{\mathbf{z}_\alpha}\right)^\alpha
\right]^{1/(d-\vert\alpha\vert)}\\
&=f(\mathbf{0})-\sum_{\alpha\in \Delta(f)^{<d}}(\sum_{i=1}^n z_{\alpha,i}+(d-\vert\alpha\vert)\left[
\left(\frac{f_\alpha}{d}\right)^{d}\,\left(\frac{\alpha}{\mathbf{z}_\alpha}\right)^\alpha
\right]^{1/(d-\vert\alpha\vert)})\\
&-\sum_{\alpha\in \Delta(f)^{=d}}(\sum_{i=1}^n z_{\alpha,i})-\sum_{i=1}^n |f_{d,i}|.
\end{align*}
We \it claim \rm that, for each $\alpha \in \Delta(f)^{<d}$, the minimum value of
\[
	\sum_{i=1}^n z_{\alpha,i}+(d-\vert\alpha\vert)\left[
	 \left(\frac{f_\alpha}{d}\right)^{d}\,\left(\frac{\alpha}{\mathbf{z}_\alpha}\right)^\alpha
	\right]^{1/(d-\vert\alpha\vert)}
\]
subject to $z_{\alpha,i}\ge 0$ and $z_{\alpha,i}=0$ iff $\alpha_i=0$ is $|f_{\alpha}|$; and that, for each
$\alpha \in \Delta(f)^{=d}$, the minimal value of
\[
	\sum_{i=1}^n z_{\alpha,i}
\]
subject to $\left(\frac{f_\alpha}{d}\right)^{d}\left(\frac{\alpha}{\mathbf{z}_\alpha}\right)^\alpha \, = \, 1$,
$z_{\alpha,i}\ge 0$ and $z_{\alpha,i}=0$ iff $\alpha_i=0$ is also equal to $|f_{\alpha}|$.
It follows from the claim that the maximum value of (\ref{obj}) is equal to
\[
	s(f,\mathbf{g}) = f(\mathbf{0})-\sum_{\alpha\in \Delta(f)}|f_{\alpha}|-\sum_{i=1}^n |f_{d,i}| = f_{\operatorname{tr},\mathbf{N}}.
\]
In proving the claim, one can reduce first to the case where each $\alpha_i$ is strictly positive. The claim, in this case,
is a consequence of the following lemma.
\end{proof}
\begin{lemma}
Suppose $\alpha_i>0$, $i=1,\dots,n$.

(1) For $|\alpha|<d$, the minimum value of
\[
	\sum_{i=1}^n z_i + (d-|\alpha|)\left[(\frac{f_{\alpha}}{d})^d\prod_{i=1}^n \frac{\alpha_i^{\alpha_i}}{z_i^{\alpha_i}}\right]^{1/(d-|\alpha|)}
\]
on the set $(0,\infty)^n$ is equal to $|f_{\alpha}|$. The minimum occurs at $z_i = \alpha_i\cdot \frac{|f_{\alpha}|}{d}$, $i=1,\dots,n$.

(2) For $|\alpha|=d$, the minimum value of $\sum_{i=1}^n z_i$ subject to $z_i>0$ and
$(\frac{|f_{\alpha}|}{d})^d\cdot \prod_{i=1}^n\frac{\alpha_i^{\alpha_i}}{z_i^{\alpha_i}} = 1$ is equal to $|f_{\alpha}|$.
The minimum occurs at $z_i = \alpha_i\cdot \frac{|f_{\alpha}|}{d}$, $i=1,\dots,n$.
\end{lemma}
\begin{proof} the optimization problem in (1) is equivalent to the problem of minimizing the function
$\sum_{i=1}^{n+1} z_i$ subject to $z_i>0$ and
$(\frac{|f_{\alpha}|}{d})^d\cdot \prod_{i=1}^n\frac{\alpha_i^{\alpha_i}}{z_i^{\alpha_i}}\cdot \frac{(d-|\alpha|)^{d-|\alpha|}}{z_{n+1}^{d-|\alpha|}} = 1$.
In this way, (1) reduces to (2). The proof of (2) is straightforward, e.g., making the change in variables
$w_i = \frac{z_i d}{\alpha_i |f_{\alpha}|}$ we are reduced to minimizing $\sum_{i=1}^n \alpha_iw_i$ subject to
$\prod_{i=1}^n w_i^{\alpha_i}=1$. Using the relation between the arithmetic and geometric mean yields
\[
	\frac{\sum_{i=1}^n \alpha_iw_i}{d}\ge \root d \of{\prod_{i=1}^n w_i^{\alpha_i}}= 1,
\]
i.e., $\sum_{i=1}^n \alpha_iw_i \ge d$. On the other hand, if we take $w_i=1$, then $\sum_{i=1}^n \alpha_iw_i = |\alpha|=d$.
Thus the minimum occurs at $w_i=1$, i.e., $z_i = \alpha_i \frac{|f_{\alpha}|}{d}$, $i=1,\dots,n$, and the minimum value of
$\sum_{i=1}^n z_i$ is $\sum_{i=1}^n \alpha_i \frac{|f_{\alpha}|}{d} = |\alpha|\frac{|f_{\alpha}|}{d} = |f_{\alpha}|$.
\end{proof}
\begin{rem} \

(1) Suppose $I_1,\dots,I_{\ell}$ and $J_1,\dots,J_m$ are partitions of $\{ 1,\dots,n\}$ with $I_1,\dots,I_{\ell}$ finer
than $J_1,\dots,J_m$,
\[
	G(\lambda) = f-\sum_{p=1}^{\ell}\lambda_p \left(1-\sum_{i\in I_p} (\frac{x_i}{N_i})^d\right), \ H(\mu) =
	f-\sum_{q=1}^m \mu_q\left(1-\sum_{i\in J_q}(\frac{x_i}{N_i})^d\right).
\]
One checks that if $\mu_q = \sum_{I_p\subseteq J_q} \lambda_p,$ then $G(\lambda)_{\operatorname{gp}} \le H(\mu)_{\operatorname{gp}}$.
It follows that $s(f,\mathbf{g}) \le s(f,\mathbf{h})$ where
\[
	\mathbf{g} = \left(1- \sum_{i\in I_1}(\frac{x_i}{N_i})^d,\dots,1- \sum_{i\in I_{\ell}}(\frac{x_i}{N_i})^d\right),
	\quad \mathbf{h} = \left(1- \sum_{i\in J_1}(\frac{x_i}{N_i})^d,\dots,1- \sum_{i\in J_m}(\frac{x_i}{N_i})^d\right).
\]

(2) Similarly, one checks that if
\[
	H(\lambda)= f-\lambda(1-\sum_{i=1}^n (\frac{x_i}{N_i})^d), \ I(\mu) = f-\sum_{j=1}^n \mu_j(\frac{1}{n}-(\frac{x_j}{N_j})^d)
\]
where $\mu_j = \lambda$, $j=1,\dots,n$, then $H(\lambda)=I(\mu)$. It follows that $s(f,\mathbf{h}) \le s(f,\mathbf{i})$ where
\[
	\mathbf{h}=\left(1-\sum_{i=1}^n(\frac{x_i}{N_i})^d\right),\ \mathbf{i}=\left(\frac{1}{n}-(\frac{x_1}{N_1})^d,\dots,\frac{1}{n}-(\frac{x_n}{N_n})^d\right).
\]

(3) In particular, (1) and (2) imply
\[
	s(f,\mathbf{g})\le s(f,\mathbf{h}) \le s(f,\mathbf{i}),
\]
with $\mathbf{g}$ as in Theorem \ref{hypercube case}, $\mathbf{h}$ and $\mathbf{i}$ as in (2). Observe also that, by
Theorem \ref{hypercube case}, $f_{\operatorname{tr},\mathbf{N}} \le s(f,\mathbf{g})$ and
$f_{\operatorname{tr},\mathbf{N}/\root d \of{n}} \le s(f,\mathbf{i})$ with equality holding if $f_{d,i}\le 0$, $i=1,\dots,n$.
This clarifies to some extent an observation made in \cite[Section 3]{gha-lass-mar}.
\end{rem}

\end{document}